\documentclass[12pt,reqno]{amsart}
\usepackage{amsfonts,amssymb,amscd,amsmath,enumerate,verbatim,calc}

\usepackage{tikz}
\usepackage{graphics}

\headsep=1cm \numberwithin{equation}{section}
\tikzstyle{vertex}=[circle, draw, inner sep=0pt, minimum size=4pt]

\tikzstyle{vertex}=[circle, draw, inner sep=0pt, minimum size=4pt] 

\newtheorem{theorem}{Theorem}[section]
\newtheorem{lemma}[theorem]{Lemma}
\newtheorem{proposition}[theorem]{Proposition}
\newtheorem{conjecture}[theorem]{Conjecture}

\theoremstyle{definition}
\newtheorem{definition}[theorem]{Definition}
\theoremstyle{remark}
\newtheorem{remark}[theorem]{Remark}
\newtheorem{example}[theorem]{Example}

\begin{document}

\title[]{On the Regularity of Dominant and Almost Complete Intersection Monomial Ideals}

\author{Amir Mafi}
\address{Department of Mathematics, University Of Kurdistan, P.O. Box: 416, Sanandaj, Iran.}

\email{a\_mafi@ipm.ir}

\author{Rando Rasul Qadir}
\address{Department of Mathematics, University Of Kurdistan, P.O. Box: 416, Sanandaj, Iran.}
\email{rando.qadir@univsul.edu.iq}



\begin{abstract} 
Let $R = k[x_1,\ldots,x_n]$ be a polynomial ring in $n$ variables over a field $k$, and let $I$ be a monomial ideal of $R$. If $I$ is an almost complete intersection, then we provide an explicit formula for the Castelnuovo-Mumford regularity of $I$ in terms of the powers of the dominant variables appearing in the regular sequence contained in $G(I)$ of length $|G(I)|-1$, where $G(I)$ is the set of minimal monomial generators of $I$. Furthermore, if $I$ is a dominant ideal or an almost complete intersection ideal, then we show that $\operatorname{reg}(\overline{I}) \leq \operatorname{reg}(I),$ where $\overline{I}$ denotes the integral closure of $I$. This provides a positive answer to the Küronya-Pintye conjecture for these two classes of monomial ideals. In addition, we give some examples to clarify these results. 
\end{abstract}

\maketitle

\section*{Introduction}
Throughout this paper, we assume that $R=k[x_1,\ldots,x_n]$ is the polynomial ring in $n$ variables over a field $k$ and $I$ is a monomial ideal of $R$. We denote the minimal set of generators of $I$ by $G(I)$ and $\mu(I)=|G(I)|$. 

This paper is motivated by the following conjecture, which was proposed in \cite{R10}:
\begin{conjecture}\label{conj}
If $I$ is a homoginuous ideal in the polynomial ring $R$. Then $\operatorname{reg}(\bar{I})\leq \operatorname{reg}(I)$, where $\bar{I}$ and $\operatorname{reg}(I)$ stand for the integral closure and Castelnuovo-Mumford regularity of the ideal $I$ respectively.
\end{conjecture}

This conjecture was recently established for $n=2,3$ in \cite{R3,R11}. Furthermore, in \cite{R3}, 
it was shown that the conjecture holds whenever $I$ is a complete intersection monomial ideal, a stable monomial ideal, a Gorenstein ideal of height $3$, or $\mathfrak{m}$-primary monomial ideal, where $\mathfrak{m}=(x_{1},\ldots,x_{n})$. On the other hand, the author of \cite{R4} recently constructed a counterexample demonstrating that conjecture generally has a negative answer for $n=4$. Nevertheless, it is of interest to determine additional classes of monomial ideals for which the conjecture holds. This is partly motivated by the fact that the Castelnuovo–Mumford regularity is a fundamental homological invariant in both computational and algebraic geometry, see for example \cite{R1, R13, R12, R14}. Moreover, computing the Castelnuovo–Mumford regularity is generally a difficult problem and typically requires sophisticated homological techniques.

Consider the $\mathbb{Z}$-graded finitely generated $R$-module $M$. Then the total Betti number of $M$ in homological degree $i$ is denoted by $\beta_{i}(M)$ and defined as $\dim_{k}Tor_{i}^{R}(M,k)$. The Betti number of $M$ in the homological degree $i$ and the external degree $j$ is denoted by $\beta_{i,j}(M)$ and defined as $\dim_{k}Tor_{i}^{R}(M,k)_{j}$. These graded Betti numbers appear in the graded minimal free resolution of $M$, which is a complex that has the following representation:

\begin{equation*}
0\to \bigoplus_{j\in\mathbb{N}}R(-j)^{\beta_{s,j}(M)}\to\cdots\to \bigoplus_{j\in\mathbb{N}}R(-j)^{\beta_{1,j}(M)}\to R\to M\to 0
\end{equation*}
such that $R(-j)$ refers to the graded free $R$-module which is generated by an element of degree $j$. For further background on chain complexes, free resolutions, and related topics, we refer the reader to \cite{R5,R4}.
The Castelnuovo-Mumford regularity of $M$ is defined as follows:
\begin{equation*}
\operatorname{reg}(M)=\operatorname{max}\{j-i:\beta_{i,j}(M)\neq 0\}.
\end{equation*}

Now, we recall the integral closure of a monomial ideal; for more details, see \cite{R15,R4}.  Suppose $I=(u_{1},\cdots,u_{q})$ is a monomial ideal in $R$ and $a_{i}\in \mathbb{N}^{n}$ is the exponent vector of the generator $u_{i}$ of $I$. We denote the set of convex combinations of the exponent vectors of the monomial ideal $I$ by $\operatorname{conv}(a_{1},\ldots,a_{q})$. The Newoton's polyhedron set of the monomial ideal $I$ is denoted by $NP(I)$ which is the convex hull of all the exponent vectors $a_{1},\ldots,a_{q}$. The extreme points of the Newton's polyhedron set of the monomial ideal $I$ is the minimal set $E(I)\subseteq G(I)$ with $NP(I)=\operatorname{conv}(E(I))+\epsilon$ such that $\epsilon\in \mathbb{R}_{\geq 0}^{n}$. According to \cite[Proposition 12.1.4]{R4}, we have $\bar{I}=\left(x^{\lceil a\rceil}|\text{ } a\in \operatorname{conv}(a_{1},\ldots,a_{q})\right)=\left(u\text{ }|\text{ }u^{k} \in I^{k},\text{ for some positive integer } k\geq 1\right)$.

The main aim of this paper is to study the Castelnuovo-Mumford regularity for almost complete intersection monomial ideals and we show that Conjecture \ref{conj} is true for the dominant and the class of almost complete intersection monomial ideals. Furthermore, we give some examples to clarify these results. For unexplained terminology and background material, we refer the reader to  \cite{R5,R15,R4}.

\section{The results}

In this section, we investigate the regularity of an almost complete intersection monomial ideal. First of all, we recall some basic definitions and results in this area. A monomial ideal $I$ of $R$ is said to be almost complete intersection (resp. complete intersection), if its minimal set of generating has the following form $G(I)=\{u_{1},...,u_{q},u_{q+1}\}$, where $u_{1},...,u_{q}$ (resp. $u_{1},...,u_{q},u_{q+1}$)  form a regular sequence. The following definitions were introduced in \cite{R1}.

\begin{definition}
A monomial ideal $ I \subseteq R $ is called a dominant ideal if its minimal generating set of monomials, $ G(I) $, satisfies the condition that for each monomial $ u \in G(I) $, there is at least one variable $ x_j $ such that the exponent of $ x_j $ in $ u $ is strictly greater than the exponent of $ x_j $ in every other monomial in $G(I)$.
\end{definition}

\begin{definition}
A monomial ideal $ I $ in the ring $ R $ is called semidominant if among generators, there is a monomial that does not have a variable with an exponent strictly greater than the corresponding exponent in all other generators.
\end{definition}

For example, $I=(x_1^3x_2,x_2^2x_3,x_3x_4)$ and $J=(x_1^2x_2^2,x_2^2x_3^2,x_1x_3)$ are dominant and semidominant ideals, respectively. In addition, the ideal $L=(x_1^2x_2^2,x_1^2x_3^2,x_2^2x_3^2)$ is neither dominant nor semidominant.

Suppose that $I=(u_{1},...,u_{q})$ is a monomial ideal in the ring $R$. Then the Taylor resolution of $R/I$ is a chain complex of the form
\begin{eqnarray}
\mathbb{T}_I: 0\to F_{q}\to F_{q-1}\to\ldots\to F_{2}\to F_{1}\to F_{0}\to R/I\to 0,
\end{eqnarray}
where $F_{i}=\oplus_{j\in \mathbb{Z}}F_{ij}$ is a free $R$-module in homological degree $i$ and internal degree $j$ and the differential map $d_{i}:F_{i}\to F_{i-1}$ is defined as follows:
\begin{equation}
d_{i}(e_{j_{1}}\land...\land e_{j_{i}})=\sum_{1\leq p \leq i}(-1)^{p-1}\frac{\operatorname{lcm}(u_{j_{1}},\ldots,u_{j_{i}})}{\operatorname{lcm}(u_{j_{1}},\ldots,\hat{u}_{j_{p}},\ldots,u_{j_{i}})} e_{j_{1}}\land\ldots \hat{e}_{j_{p}}\ldots\land e_{j_{i}},
\end{equation}
for each basis element $e_{j_{1}}\land\ldots\land e_{j_{i}}$ of the free $R$-module $F_{i}$ in the homological degree $i$ and the internal degree $\operatorname{lcm}(u_{j_{1}},\ldots,u_{j_{i}})$. The notations $\hat{u}_{j_{p}}$ and $\hat{e}_{j_{p}}$ referring to $u_{j_{p}}$ and $e_{j_{p}}$ are omitted  in $\operatorname{lcm}(u_{j_{1}},\ldots,u_{j_{i}})$ and the base element $e_{j_{1}}\land\ldots\land e_{j_{i}}$, respectively (see for more details \cite{R5}).

In the following theorem, we recall the construction of Lyubznik's resolution. Lyubznik in \cite{R7} provided that for any arrangement of the generators of a given monomial ideal, there is a finite free resolution, which contains a minimal free resolution of the monomial ideal and it is more accurate compared with the Taylor resolution. More details and examples on Lyubznik's resolution can be found in \cite[Section 6]{R8}.

\begin{theorem}[Lyubznik]\label{lyubznik}
If  $I=(u_{1},...,u_{q})$ is a monomial ideal in the polynomial ring $R$. Then there is a free resolution $L$ of $R/I$, for which it is a subcomplex of the Taylor complex $\mathbb{T}_{I}$ and it is generated in every dimension $j$ by elements $e_{1}\land\cdots\land e_{i_{j}}$  such that $u_{r}$ does not divide $\operatorname{lcm}(u_{i_{t}}, u_{i_{t+1} },\cdots,u_{i_{j }})$ for all $t < j$ and $r < i_{t}$.
\end{theorem}

In this section, we first recall the Scarf complex of a monomial ideal $I=(u_{1},\ldots,u_{q})$. Let $\Gamma_{I}=\{\sigma\subseteq G(I)| \operatorname{lcm}(\sigma)\neq \operatorname{lcm}(\tau),\text{ for every }\sigma\neq\tau\subseteq G(I)\}$. $\Gamma_{I}$ is said to be a Scarf complex. The simplicial chain complex induced by $\Gamma_{I}$ is indicated by $\mathbb{F}_{\Gamma_{I}}$ and is called the Scarf chain complex. The basis elements of the chain complex $\mathbb{F}_{\Gamma_{I}}$ in homological degree $i$ are of the form $e_{j_{1}}\land\ldots\land e_{j_{i}}$, where $u_{j_{1}},\ldots, u_{j_{i}}$ is a face in $\Gamma_{I}$. This complex is introduced and investigated by Scarf \cite{R17} and after that Bayer, Peeva and Sturmfels reintroduced it as Scarf complex in \cite{R16}. They provided that the Scarf complex always contains in minimal free resolution of a monomial ideal. Furthermore, they investigated some special types of monomial ideal such that the Scarf complex forms a minimal free resolution. Alesandroni in \cite{R1} provided that the dominant and semidominant monomial ideals are of those classes of ideals for which the scarf complex forms a minimal free resolution for them. 

According to \cite[Theorem 4.8]{R1}, for a dominant monomial ideal $I$ of $R$ with the minimal generating set $G(I)=\{u_{1},\ldots,u_{q}\}$. If $h=\operatorname{deg}(\operatorname{lcm}(u_{1},\ldots,u_{q}))$, then $\operatorname{reg}(R/I)=h-q$ and it is obvious,
\begin{align}
\operatorname{deg}(\operatorname{lcm}(u_{1},...,u_{q}))&=\sum a_{i}, \text{ where }a_{i}=\max\{\deg_{x_{i}}u: u\in G(I)\}.
\end{align}

Consequently, 
\begin{equation}
\label{regularity dominant}
\operatorname{reg}(I)=\sum_{i=1}^q a_{i}-q+1.
\end{equation}
On the other hand, if $I$ is a semidominant monomial ideal in $R$ with minimal generating set $G(I)=\{u_1,\ldots,u_q,v\}$ and $u_{1},\ldots,u_{q}$ form a dominant set and $v$ is nondominant monomial, then by \cite[Corollary 6.10]{R1}, we obtain that

\begin{equation}\label{eq3}
\operatorname{reg}(R/I)=\max\{\operatorname{deg}(\operatorname{lcm}(\sigma))-|\sigma|:\sigma\subset G(I), v\in\sigma\quad and\quad \sigma\quad is\quad dominant\}.
\end{equation}

\begin{remark}\label{remark1.4}

From \cite[Proposition 6.3]{R1}, we see that the almost complete intersection monomial ideal is dominant or semidominant. Let $I=(u_1,\ldots,u_q,u_{q+1})$ be an almost complete intersection monomial ideal. Then $I$ has one of the following form:
\begin{itemize}
    \item $I$ is a dominant monomial ideal.
    \item there exist $2 \leq s \leq q$ with $u_{q+1} \mid \operatorname{lcm}\left(u_1, \ldots, u_s\right)$ and $u_{q+1} \nmid \operatorname{lcm}(\sigma)$ for any $\sigma \subset\left\{u_1, \ldots, u_s\right\}$ and $|\sigma|<s$.
\end{itemize}

If $I$ satisfies the first case. From \cite[Corollary 4.9]{R1}, we have $\beta_{i}(R/I)=\binom{q+1}{i}$. If $I$ satisfies the second case, then by \cite[Corollary 6.4]{R1}, we obtain the Scarf resolution is the minimal free resolution of $I$. That means; the minimal free resolution of $I$ can be obtained from the Taylor resolution $\mathbb{T}_{I}$, by standard cancellation of those basis elements of $\mathbb{T}_{I}$ contains $\{u_{1}, \ldots, u_{s}\}$. From \cite[Theorem 1.14]{R6}, we have $\beta_{i}(R/I) = \binom{q+1}{i}$ for $i=0,1,\ldots,s-1$ and $\beta_{i}(R/I) = \binom{q+1}{q+1-i} - \binom{q+1-s}{q+1-i}$ for $i=s,\ldots,q$. Hence, the graded Betti numbers in the graded minimal free resolution of $I$ depend on whether the basis elements of $\mathbb{T}_{I}$ contain $\{u_{1},\ldots, u_{s}\}$ or not, as well as the degrees of the generators of $I$. In the following proposition, which follows readily from the above argument, we describe the graded minimal free resolution of an almost complete intersection monomial ideal.
\end{remark}

\begin{proposition}
Let $I=(u_1,\ldots,u_q,u_{q+1})$ be an almost complete intersection monomial ideal of $R$. If $I$ is dominant, then the graded minimal free resolution of $R/I$ has the form:
$$0\to F_{q+1}\to F_{q} \to \cdots \to F_{1} \to F_{0} \to R/I \to 0$$
where $F_{i}=\bigoplus_{\tau\subseteq G(I),|\tau|=i}R(-\operatorname{lcm}(\tau))$, for $i=0,\ldots,q+1$. If $I$ satisfies the second case of Remark \ref{remark1.4}, then the graded minimal free resolution of $R/I$ has the form: 
$$0\to F_{q} \to \cdots \to F_{1} \to F_{0} \to R/I \to 0$$
where $F_{i}=\bigoplus_{\tau\subseteq G(I),|\tau|=i,u_{j}\notin\tau}R(-\operatorname{lcm}(\tau))$, for $i=0,\ldots,q$ and $j=1,\ldots,s.$
\end{proposition}

By the above arguments, we obtain the regularity of an almost complete intersection monomial ideal as the following:

\begin{theorem} \label{th1.3}
If $I$ is an almost complete intersection monomial ideal of the second case of Remark \ref{remark1.4}. Then
\begin{equation}\label{sdreg}
\operatorname{reg}(I)=\operatorname{reg}(I')-\min\left\{\deg\left(\frac{u_{j}}{\gcd(u_{j},u_{q+1})}\right),\quad j=1,\ldots,s\right\}.
\end{equation}
If the monomial $u_{q+1}$ does not divide the least common multiple of any subset $\sigma$ of $G(I)$ and there exists $1\leq h\leq q$ with $\operatorname{Supp}(u_{j})\cap \operatorname{Supp}(u_{q+1})$ which is not empty for $j=1,\ldots,h$. Then  
\begin{equation}\label{dreg}
\operatorname{reg}(I)=\sum_{l=1}^{q+1}\operatorname{deg}(u_{l})-q-\sum_{i=1}^{h}\operatorname{deg}(\gcd(u_{i},u_{q+1}).
\end{equation}
\end{theorem}
\begin{proof}
For the first case, suppose that $I=(u_1,\ldots,u_q,u_{q+1})$ and $u_{q+1}|\operatorname{lcm}(u_{1},\ldots,u_{s})$. Now, for each $1\leq k \leq s$ we want to show that
\begin{equation}\label{eq5}
\operatorname{deg}(\operatorname{lcm}(u_{1},\ldots,u_{k-1},u_{q+1},u_{k+1},\ldots,u_{s},\ldots,u_{q}))=\sum_{i=1}^{q}\operatorname{deg}(u_{i})-\deg\left(\frac{u_{k}}{\gcd(u_{k},u_{q+1})}\right).
\end{equation}
Since $u_{q+1}|\operatorname{lcm}(u_1,\ldots,u_s )$ and $u_{q+1}\nmid \operatorname{lcm}(u_1,\ldots,\bar{u_{l}},\ldots,u_s )$, for $l=1,\ldots,s$, then $\operatorname{Supp}(u_{k})\cap \operatorname{Supp}(u_{q+1})$ is not empty. Without loss of generality, suppose that $\operatorname{Supp}(u_{k})\cap \operatorname{Supp}(u_{q+1})=\{x_{1},\ldots,x_{r}\}$ and $a_{i}$ is the degree of $x_{i}$ in the monomial $u_{q+1}$. Since $u_1,\ldots,u_q$ is a regular sequence, we have the following
\begin{eqnarray}
&&\operatorname{deg}(\operatorname{lcm}(u_{1},\ldots,u_{k-1},u_{q+1},u_{k+1},ßldots,u_{s},\ldots,u_{q}))=\nonumber\\&&\operatorname{deg}(u_{1})+\ldots+\operatorname{deg}(u_{k-1})+\operatorname{deg}(u_{k+1})+\ldots+\operatorname{deg}(u_{q})+(a_{1}+\ldots+a_{r})\nonumber\\&&
=\operatorname{deg}(u_{1})+\ldots+\operatorname{deg}(u_{q})-\operatorname{deg}(u_{k})+(a_{1}+\ldots+a_{r})\nonumber\\&&
=\operatorname{deg}(u_{1})+\ldots+\operatorname{deg}(u_{q})-(\operatorname{deg}(u_{k})-(a_{1}+\ldots+a_{r})).\label{eq6}
\end{eqnarray}
Since $\operatorname{Supp}(u_{k})\cap \operatorname{Supp}(u_{q+1})=\{x_{1},\ldots,x_{r}\}$, then $\gcd(u_{k},u_{q+1})=x_{1}^{a_{1}}\ldots x_{r}^{a_r}$, where $a_{i}$ is the degree of $x_{i}$ in the monomial $u_{q+1}$. Then $\deg\left(\frac{u_{k}}{\gcd(u_{k},u_{q+1})}\right)=\operatorname{deg}(u_{k})-(a_1+\ldots+a_r)$, by putting this in equation (\ref{eq6}) we obtain (\ref{eq5}). Since $u_{1},\ldots,u_{q}$ is a regular sequence, it is obvious that for any $\tau\subseteq G(I)$ with $u_{q+1}\in \tau$ and $\tau$ is dominant, there exists $\sigma\subseteq G(I)$ with $u_{q+1}\in \sigma$, $|\tau|<|\sigma|$ and $\sigma$ is dominant such that $\operatorname{deg}(\operatorname{lcm}(\tau))-|\tau|\leq\operatorname{deg}(\operatorname{lcm}(\sigma))-|\sigma|$. According to (\ref{eq3}), we obtain  
\begin{align*}
\operatorname{reg}(R/I)&=\max\{\operatorname{deg}(\operatorname{lcm}(\sigma))-|\sigma|:\sigma\subset G(I), u_{q+1}\in\sigma, |\sigma|=q\quad and\quad \sigma\quad is\quad dominant\}\nonumber\\&
=\max\left\{\operatorname{deg}(\operatorname{lcm}(u_{1},\ldots,u_{k-1},u_{q+1},u_{k+1},\ldots,u_{s},\ldots,u_{q}))-q: k=1,\ldots,s\right\}\nonumber\\&
=\max\left\{\sum_{i=1}^{q}\operatorname{deg}(u_{i})-q-\deg\left(\frac{u_{k}}{\gcd(u_{k},u_{q+1})}\right): k=1,\ldots,s\right\}\nonumber\\&
=\sum_{i=1}^{q}\operatorname{deg}(u_{i})-q-\min\left\{\deg\left(\frac{u_{j}}{\gcd(u_{j},u_{q+1})}\right),\quad j=1,\ldots,s\right\}.
\end{align*}
Hence, 
\begin{align*}
\operatorname{reg}(I)&=\sum_{i=1}^{q}\operatorname{deg}(u_{i})-q+1-\min\left\{\deg\left(\frac{u_{j}}{\gcd(u_{j},u_{q+1})}\right),\quad j=1,\ldots,s\right\}\\&
=\operatorname{reg}(I')-\min\left\{\deg\left(\frac{u_{j}}{\gcd(u_{j},u_{q+1})}\right),\quad j=1,\ldots,s\right\}.
\end{align*}
For the second case, if the monomial $v$ does not divide the least common multiple of any subset $\sigma$ of $G(I)$, then the monomial ideal is dominant and the result follows by (\ref{regularity dominant}).
\end{proof}

The following examples are devoted to compute the regularity of an almost complete intersection monomial ideal in the first and the second case of Theorem \ref{th1.3}, respectively.

\begin{example}
Consider the polynomial ring $R=k[x_{1},\ldots,x_{10}]$ and 

$$I=(x_{1}^{5} x_{2}^{4},x_{3}^{6}x_{4}^{3}x_{5}^{3},x_{6}^{7}x_{7}^{2},x_{8}^{5}x_{9}^{4},x_{10}^{3},x_{1}^{3}x_{3}^{4}x_{6}^{2}).$$
It is clear that $I$ is an almost complete intersection monomial ideal with the complete intersection part $I^{'}=(x_{1}^{5} x_{2}^{4},x_{3}^{6}x_{4}^{3}x_{5}^{3},x_{6}^{7}x_{7}^{2},x_{8}^{5}x_{9}^{4},x_{10}^{3})$. Here, $q=5$ and $u_{6}|\operatorname{lcm}(u_{1},u_{2},u_{3})$. Then $\operatorname{reg}(I^{'})=9+12+9+9+3-5+1=38$. Now, $\frac{u_{1}}{gcd(u_{1},u_{6})}=x_{1}^{2}x_{2}^{4}$, $\frac{u_{2}}{gcd(u_{2},u_{6})}=x_{3}^{2}x_{4}^{3}x_{5}^{3}$ and $\frac{u_{3}}{gcd(u_{3},u_{6})}=x_{6}^{5}x_{7}^{2}$. Then 
$$\min\left\{\deg\left(\frac{u_{j}}{\gcd(u_{j},u_{q+1})}\right), j=1,2,3\right\}=6.$$ 

Consequently, we obtain $\operatorname{reg}(I)=\operatorname{reg}(I^{'})-6=38-6=32$. We achieve the same result by applying \textit{Macaulay2} \cite{R9}:

\begin{verbatim}
i1: R = QQ[x_1..x_10];
i2: I = monomialIdeal(x_1^5*x_2^4, x_3^6*x_4^3*x_5^3, x_6^7*x_7^2,
    x_8^5*x_9^4, x_10^3, x_1^3*x_3^4*x_6^2);
i3: regularity I
o2: 32
\end{verbatim}
\end{example}

\begin{example}
In the polynomial ring $R=k[x_{1},\ldots,x_{10}]$, if we have the following almost complete intersection monomial ideal 

$$I=(x_{1}^{5} x_{2}^{4},x_{3}^{6}x_{4}^{3}x_{5}^{3},x_{6}^{7}x_{7}^{2},x_{8}^{5}x_{9}^{4},x_{10}^{3},x_{1}^{6}x_{3}^{7}x_{6}^{2}).$$
It is clear that $I$ is dominant. Here, $q=5$ and $\operatorname{Supp}(u_{j})\cap\operatorname{Supp}(u_{6})\neq \phi$, for $j=1,2,3$. Now, $\gcd(u_{1},u_{6})=x_{1}^{5}$, $\gcd(u_{2},u_{6})=x_{3}^{6}$ and $\gcd(u_{3},u_{6})=x_{6}^{2}$. Then 
$$\operatorname{reg}(I)=9+12+9+9+3+15-5-(5+6+2)=39.$$
\end{example}

\begin{verbatim}
i1: R = QQ[x_1..x_10];
i2: I = monomialIdeal(x_1^5*x_2^4, x_3^6*x_4^3*x_5^3, x_6^7*x_7^2,
    x_8^5*x_9^4, x_10^3, x_1^6*x_3^7*x_6^2);
i3: regularity I
o2: 39
\end{verbatim}

These coincide with our result.

\begin{remark}\label{remark1.5}
If $I=\left(u_{1}, \ldots, u_{q}, u_{q+1}\right)$ is an almost complete intersection monomial ideal and $u_{1}, \ldots, u_{q}$ is a regular sequence, then from the definition of the Newton polyhedron set and the fact that $u_{1}, \ldots, u_{q}$ is a regular sequence, we obtain $u_{1}, \ldots, u_{q}\in G(\bar{I})$. We know that $u_{q+1}\in \bar{I}$, since $I\subseteq \bar{I}$. If $u_{q+1}\notin G(\bar{I})$, then there exists $v\in G(\bar{I})$ with $v|u_{q+1}$. It is obvious that $v \ne u_{i}$ for $i=1,\ldots,q$. This implies $\mu(\bar{I}) \ge q+1 = \mu(I)$.
\end{remark}

\begin{lemma}\label{lemma1.5}
Suppose $I$ is an almost complete intersection monomial ideal satisfies the second case of Remark \ref{remark1.4}. Consider $u_{h}$ with $1 \leq h \leq s$ such that $\operatorname{supp}\left(u_h\right) \cap \operatorname{supp}\left(u_{q+1}\right)=\left\{x_{h_1}, \ldots, x_{h_{j^{\prime}}}\right\}$. If there is a monomial $v \in \bar{I}$ with $v=\left(x_{h_1}^{b_{h_1}} \ldots x_{h_{j^{\prime}}}^{b_{h_{j^{\prime}}}}\right)  \cdot v^{\prime\prime}$, then there exists a monomial $u \in \bar{I}$ with $u=\left(x_{h_1}^{c_{h_1}} \ldots x_{h_{j^{\prime}}}^{c_{h_{j^{\prime}}}}\right) \cdot u^{\prime}, c_{h_l} \leq b_{h_l}$, for $l=1, \ldots, j^{\prime}$ and $\operatorname{supp}\left(u^{\prime}\right) \cap \operatorname{supp}\left(u_h\right)=\phi$. 
\end{lemma}

\begin{proof}
Consider the monomial $u_h$ with $1 \leq h \leq s$, such that $u_h=\left(x_{h_1}^{a_{h_1}} \ldots x_{h_{j^{\prime}}}^{a_{h_{j^{\prime}}}}\right) \cdot\left(x_{k_1}^{a_{k_1}} \ldots x_{k_{l^{\prime}}}^{a_{k_{l^{\prime}}}}\right)$ and suppose $v \in \bar{I}$ with $v=\left(x_{h_1}^{b_{h_1}} \ldots x_{h_{j^{\prime}}}^{b_{h_{j^{\prime}}}}\right) \cdot\left(x_{k_1}^{b_{k_1}} \ldots x_{k_{l^{\prime}}}^{b_{k_{l^{\prime}}}}\right) \cdot v^{\prime}$ and $\operatorname{supp}\left(v^{\prime}\right) \cap \operatorname{supp}\left(u_h\right)=\phi$. We may consider $u_{q+1}=\left(x_{h_1}^{d_{h_1}} \ldots x_{h_{j^{\prime}}}^{d_{h_{j^{\prime}}}}\right).w$ and $\operatorname{supp}(w) \cap \operatorname{supp}\left(u_h\right)=\phi$. Since $v \in \bar{I}$, there exist $0 \leq \lambda_1, \ldots, \lambda_q, \lambda_{q+1} \leq 1$ with $\lambda_1+\cdots+\lambda_h+\cdots+\lambda_q+\lambda_{q+1}=1$ and $a_v=\lambda_1 a_{u_1}+\cdots+\lambda_h a_{u_h}+\cdots+\lambda_q a_{u_q}+\lambda_{q+1} a_{u_{q+1}} \in N P(I)$, where $a_v$ is the exponent vector of the monomial $v$ of the minimal generators of $I$. If $\alpha_t=\lambda_t$, for $t \neq h, q+1$, while $\alpha_h=0$ and $\alpha_{q+1}=\lambda_h+\lambda_{q+1}$, then $0 \leq \alpha_1, \ldots, \alpha_q, \alpha_{q+1} \leq 1$ with $\alpha_1+\cdots+\alpha_h+\cdots+\alpha_q+\alpha_{q+1}=1$, this means that $\alpha_1 a_{u_1}+\cdots+ \alpha_h a_{u_h}+\cdots+\alpha_q a_{u_q}+\alpha_{q+1} a_{u_{q+1}} \in N P(I)$. If we put $a_u=\alpha_1 a_{u_1}+\cdots+\alpha_h a_{u_h}+\cdots+\alpha_q a_{u_q}+ \alpha_{q+1} a_{u_{q+1}}$, then it is obvious that for $r=1, \ldots, l^{\prime}$ the $k_r$ components of $a_u$ are all zero and for $r^{\prime}= 1, \ldots, j^{\prime}$ the $h_{r^{\prime}}$ component of $a_u$ has the form $\left(\lambda_h+\lambda_{q+1}\right) d_{h_{r^{\prime}}}$ and satisfies $\left(\lambda_h+\lambda_{q+1}\right) d_{h_{r^{\prime}}} \leq \lambda_h a_{h_{r^{\prime}}}+\lambda_{q+1} d_{h_{r^{\prime}}}=b_{h_{r^{\prime}}}$. Hence, $u \in \bar{I}$ and has the form $u=\left(x_{h_1}^{c_{h_1}} \ldots x_{h_{j^{\prime}}}^{c_{h_{j^{\prime}}}}\right) \cdot u^{\prime}, c_{h_{r^{\prime}}} \leq b_{h_{r^{\prime}}}$, for $r^{\prime}=1, \ldots, j^{\prime}$ and $\operatorname{supp}\left(u^{\prime}\right) \cap \operatorname{supp}\left(u_h\right)=\phi$. This completes the proof.
\end{proof}

To provide a clearer and more detailed explanation of the lemma stated above, we present the following illustrative example.

\begin{example}
If $I=\left(x_1^3 x_2^4 x_3^2 x_4, x_5^4 x_6^2 x_7^2 x_8^3, x_1^2 x_2^2 x_5^3 x_6^2\right)=(u_{1},u_{2},u_{3})$ is an almost complete intersection monomial ideal in the polynomial ring $k[x_{1},\ldots,x_{8}]$. We can see that $\operatorname{supp}(u_{1})\cap \operatorname{supp} (u_{3})=\{x_{1},x_{2}\}$ and  $\operatorname{supp}(u_{2})\cap \operatorname{supp} (u_{3})=\{x_{5},x_{6}\}$. Applying \textit{Macaulay2}, we obtain

$\bar{I}=I+(x_1 x_2 x_5^4 x_6^2 x_7 x_8^2, x_1 x_2 x_3 x_4 x_5^3 x_6^2 x_7^2 x_8^3, x_1 x_2^2 x_3 x_4 x_5^3 x_6^2 x_7^2 x_8^2, x_1^2 x_2^2 x_3 x_4 x_5^2 x_6 x_7 x_8^2,$ 

$ x_1^2 x_2^3 x_3 x_4 x_5^2 x_6 x_7 x_8, x_1^3 x_2^3 x_3 x_4 x_5^2 x_6, x_1^3 x_2^3 x_3^2 x_4 x_5 x_6 x_7 x_8)$.

Now, $v=x_1^2 x_2^2 x_3 x_4 x_5^2 x_6 x_7 x_8^2\in \bar{I}$ and $\operatorname{supp}(u_{1})\subseteq \operatorname{supp}(v)$ with $\operatorname{deg}_{x_{1}}(v)=2$ and $\operatorname{deg}_{x_{2}}(v)=2$. We aim to verify the existence of a monomial with the required property $u\in \bar{I}$ such that $\operatorname{supp}(u)\cap \operatorname{supp}(u_{1})=\{x_{1},x_{2}\}$, $\operatorname{deg}_{x_{i}}(u)\leq \operatorname{deg}_{x_{i}}(v)$ for $i=1,2$. To do this, we choose $u=x_1 x_2 x_5^4 x_6^2 x_7 x_8^2\in \bar{I}$. By the same way, if $v=x_1^3 x_2^3 x_3^2 x_4 x_5 x_6 x_7 x_8\in \bar{I}$ and $\operatorname{supp}(u_{2})\subseteq \operatorname{supp}(v)$ with $\operatorname{deg}_{x_{5}}(v)=1$ and $\operatorname{deg}_{x_{6}}(v)=1$. Now, $a_{v}=(3,3,2,1,1,1,1,1)=\frac{3}{4}a_{u_{1}}+\frac{1}{4}a_{u_{2}}+0.a_{u_{3}}$, and so we take $a_{u}=\frac{3}{4}.a_{u_{1}}+0.a_{u_{2}}+(\frac{1}{4}+0).a_{u_{3}} =\frac{3}{4}(3,4,2,1,0,0,0,0)+0.(0,0,0,0,4,2,2,3)+\frac{1}{4}(2,2,0,0,3,2,0,0) = (5,4,2,1,1,1,0,0)$. Then $u=x_{1}^{5}x_{2}^{4}x_{3}^{2}x_{4}x_{5}x_{6}=x_{5}x_{6}.u_{1}\in \bar{I}$, where $\operatorname{supp}(u)\cap \operatorname{supp}(u_{2})=\{x_{5},x_{6}\}$, $\operatorname{deg}_{x_{i}}(u)\leq \operatorname{deg}_{x_{i}}(v)$ for $i=5,6$.
\end{example}

\begin{lemma}\label{lemma1.6}
Suppose that $I$ is an almost complete intersection monomial ideal satisfies the second case of Remark \ref{remark1.4} and there exists a generator $u_{i}$ for some $1 \leq i \leq s$ with $\operatorname{supp}\left(u_i\right) \cap \operatorname{supp}\left(u_{q+1}\right)=\left\{x_1, \ldots, x_j\right\}$.  If $u^{\prime}=x_1^{c_1} x_2^{c_2} \ldots x_j^{c_j} \cdot x_{i_1}^{c_{i_1}} \ldots x_{i_l}^{c_{i_l}} \cdot v \in \bar{I}$, where $v$ is a monomial with $c_k=0$, for some $1 \leq k \leq j$, then $c_1=\cdots=c_j=c_{i_1}=\cdots=c_{i_l}=0$.
\end{lemma}

\begin{proof}
Consider the fixed monomial $u_i$ for $1 \leq i \leq s$, then $\operatorname{supp}\left(u_i\right) \cap \operatorname{supp}\left(u_{q+1}\right)=\left\{x_1, \ldots, x_j\right\}$. That means $u_{q+1}=x_1^{b_1} \ldots x_j^{b_j} \cdot u$, where $u$ is a monomial and $\operatorname{supp}\left(u_i\right) \cap \operatorname{supp}(u)=\phi$. Let $u= x_{j_1}^{b_{j_1}} \ldots x_{j_k}^{b_{j_k}}, u_i=x_1^{a_1} \ldots x_j^{a_j} \cdot x_{i_1}^{a_{i_1}} \ldots x_{i_l}^{a_{i_l}}$, where $b_t \leq a_t$ for $t=1,\ldots, j$. If $u^{\prime} \in \bar{I}$ such that $u^{\prime}= x_1^{c_1} x_2^{c_2} \ldots x_j^{c_j} \cdot x_{i_1}^{c_{i_1}} \ldots x_{i_l}^{c_{i_l}}.v$ with $c_1=0$, then there exists a positive integer $k$ with ${u^{\prime}}^k \in I^k$. Therefore, $x_2^{k c_2} \ldots x_j^{k c_j} \cdot x_{i_1}^{k c_{i_1}} \ldots x_{i_l}^{k c_{i_l}} v^k \in I^k$ which is impossible, since all generators $I^k$ are divisible by $u_i^r= x_1^{r a_1} \ldots x_j^{r a_j} \cdot x_{i_1}^{r a_{i_1}} \ldots x_{i_l}^{r a_{i_l}}$ for some $0\leq r\leq k$. Thus, from $x_2^{k c_2} \ldots x_j^{k c_j} \cdot x_{i_1}^{k c_{i_1}} \ldots x_{i_l}^{k c_{i_l}} v^k \in I^k$ we deduce that $x_2^{k c_2} \ldots x_j^{k c_j} \cdot x_{i_1}^{k c_{i_1}} \ldots x_{i_l}^{k c_{i_l}} v^k$ is divisible by a generator of $\bar{I}$. Consequently, $u_i$ becomes redundant and may be omitted from this generator of $\bar{I}$. This means that if one of $c_1, \ldots, c_j$ is zero in a generator of $\bar{I}$, then $c_1=\cdots=c_j=c_{i_1}=\cdots=c_{i_l}=$ 0, as required.
\end{proof}

To clarify the above lemma and make its meaning more transparent, we present a simple illustrative example that highlights the main idea in a concrete setting.
  
\begin{example}
If $I=\left(x_1^3 x_2^4 x_3^2 x_4, x_5^4 x_6^2 x_7^2 x_8^3,x_{9}^{2}x_{10}, x_1^2 x_2^2 x_5^3 x_6^2\right)=(u_{1},u_{2},u_{3},u_{4})$ is an almost complete intersection monomial ideal in the polynomial ring $k[x_{1},\ldots,x_{10}]$. We can see that $\operatorname{supp}(u_{1})\cap \operatorname{supp} (u_{4})=\{x_{1},x_{2}\}$ and $\operatorname{supp}(u_{2})\cap \operatorname{supp} (u_{4})=\{x_{5},x_{6}\}$. Applying \textit{Macaulay2}, we obtain
\begin{eqnarray}
\bar{I}=I+&&(x_5^2 x_6 x_7 x_8^2 x_9 x_{10}, x_1 x_2 x_5^2 x_6 x_9 x_{10}, x_1 x_2 x_5^4 x_6^2 x_7 x_8^2, x_1 x_2 x_3 x_4 x_5 x_6 x_7 x_8 x_9 x_{10},\nonumber\\&& x_1 x_2 x_3 x_4 x_5^3 x_6^2 x_7^2 x_8^3, x_1 x_2^2 x_3 x_4 x_5^3 x_6^2 x_7^2 x_8^2, x_1^2 x_2^2 x_3 x_4 x_9 x_{10}, x_1^2 x_2^2 x_3 x_4 x_5^2 x_6 x_7 x_8^2,\nonumber\\&& x_1^2 x_2^3 x_3 x_4 x_5^2 x_6 x_7 x_8,x_1^3 x_2^3 x_3 x_4 x_5^2 x_6, x_1^3 x_2^3 x_3^2 x_4 x_5 x_6 x_7 x_8)\nonumber
\end{eqnarray}
We can see that if the power of $x_{1}$ or $x_{2}$ is zero in a generator $u$ of $\bar{I}$, then the power of all the variables of $\operatorname{supp}(u_{1})$ is zero in $u$. The same argument is true for $x_{5}$ and $x_{6}$.
\end{example}

The preceding two lemmas enable us to eliminate several basis elements from the Taylor resolution of $\bar{I}$ that do not contribute to its minimal free resolution.The significance of this process lies in its ability to determine the largest degree among the basis elements of minimal free resolution. As a consequence, it yields an upper bound for the Castelnuovo–Mumford regularity of $R/\bar{I}$.

\begin{remark}\label{remark1.13}
Suppose that $I$ is an almost complete intersection monomial ideal satisfies the second case of Remark \ref{remark1.4}.  Then $\bar{I}=\left(u_{1}, \ldots, u_{q}, v_{0}, v_{1}, \ldots, v_{p}, \right)$, where $v_{0} \mid u_{q+1}$. Therefore, we can consider the following two cases:

\begin{itemize}
\item For every $i=1,\ldots,s$, there exists $l=0,\ldots,p$ such that $\operatorname{Supp}(u_{i})\cap \operatorname{Supp}(u_{q+1})\not\subseteq \operatorname{supp}(v_{l})$, or 
\item If $\operatorname{Supp}(u_{i})\cap \operatorname{Supp}(u_{q+1}) \subseteq \operatorname{supp}\left(v_{l}\right)$, for every $l=0,1, \ldots, p$, then there exists a monomial $v^{\prime \prime}:=x_{1}^{d_{1}} \cdots x_{j}^{d_{j}} \cdot v^{\prime}\in \{v_{0},\cdots,v_{p}\}$ for a monomial $v^{\prime}$ such that $\operatorname{supp}\left(u_{i}\right) \cap \operatorname{supp}\left(v^{\prime}\right)=\varnothing$ and $d_{l}$, for $l=1,\ldots,j$, is the minimum power of $x_{l}$ appearing in the generators of $\bar{I}$. 
\end{itemize}
\end{remark}

\begin{theorem}\label{th1.14}
Suppose that $I=\left(u_{1}, \ldots, u_{q}, u_{q+1}\right)$ is an almost complete intersection monomial ideal satisfied in Remark \ref{remark1.13}. Then
\begin{eqnarray}
&&\operatorname{reg}\left(R/\bar{I}\right) \leq\nonumber\\&& \max \left\{\operatorname{deg}\left(\operatorname{lcm}\left(\left\{u_{1}, \ldots, u_{i-1}, u_{q+1}, u_{i+1}, \ldots, u_{s}, \ldots, u_{q}\right\}\right)\right)-q ; i=1, \ldots, s\right\}.\nonumber
\end{eqnarray}

\end{theorem}

\begin{proof}
According to Remark \ref{remark1.5}, since $u_{1}, \ldots, u_{q}$ is a regular sequence, they appear in the minimal generating set of $\bar{I}$. Then we can rearrange the integral closure of $I$ as $\bar{I}=\left( v_{0}, v_{1}, \ldots, v_{p}, u_{1}, \ldots, u_{q}\right)$, where $v_{0} \mid u_{q+1}$ and the degree of $v_{i}$ does not exceed the maximum degree of $u_{j}$ for $i=1, \ldots, p$ and $j=1, \ldots, q+1$. Now, suppose that $\sigma$ is any subset of $G(\bar{I})$ such that $\left\{u_{1}, \ldots, u_{s}\right\} \subseteq \sigma$. Then $v_{0}\mid \operatorname{lcm}(\sigma)$. Therefore, from the Lyubiznik resolution we obtain that $\sigma$  does not form a base in the minimal free resolution of $\bar{I}$. In the second step, we look for those subsets $A$ of $G(\bar{I})$ such that $u_{i} \notin \sigma$, for $1 \leq i \leq s$. Let $u_{i}=x_{1}^{a_{1}} \ldots x_{j}^{a_{j}} \cdot x_{i_{1}}^{a_{i_{1}}} \ldots x_{i_{l}}^{a_{i_{l}}}$ and $u_{q+1}=x_{1}^{b_{1}} \ldots x_{j}^{b_{j}} \cdot u$, where $u$ is a monomial element and $\operatorname{supp}\left(u_{i}\right) \cap \operatorname{supp}(u)=\phi$. If there exists a monomial $v_{l} \in G(\bar{I})\backslash\{u_{1},...,u_{i-1},u_{i+1}...,u_{q}\}$ such that $x_{k} \notin \operatorname{supp}\left(v_{l}\right)$, for some $1 \leq k \leq j$, then from Lemma \ref{lemma1.6} we have $\operatorname{supp}\left(u_{i}\right) \cap \operatorname{supp}\left(v_{l}\right)=\phi$. Then $v_{l}|\operatorname{lcm}\left(u_{1}, \ldots, u_{i-1}, u_{i+1}, \ldots, u_{s}, \ldots, u_{q}\right)$. This implies that if $\sigma$ is any subset of $G(\bar{I})$ such that $\left\{u_{1}, \ldots, u_{i-1}, u_{i+1}, \ldots, u_{s}, \ldots, u_{q}\right\} \subseteq \sigma$, then $\sigma$ does not form a base in the minimal free resolution of $\bar{I}$. If there is no monomial $v_{l} \in G(\bar{I})\backslash\{u_{1},\ldots,u_{i-1},u_{i+1}...,u_{q}\}$ such that $x_{k} \notin \operatorname{supp}\left(v_{l}\right)$, for some $1 \leq k \leq j$, then from Lemma \ref{lemma1.5} and Lemma \ref{lemma1.6} we have $\left\{x_{1}, \ldots, x_{j}\right\} \subseteq \operatorname{supp}\left(v_{l}\right)$, for every $l=0,1, \ldots, p$ and there exists a monomial $v^{\prime \prime}:=x_{1}^{d_{1}} \cdots x_{j}^{d_{j}} \cdot v^{\prime}$ for a monomial $v^{\prime}$ such that $\operatorname{supp}\left(u_{i}\right) \cap \operatorname{supp}\left(v^{\prime}\right)=\varnothing$ and $d_{l}$ is the minimum power of $x_{l}$ appearing in the generators of $\bar{I}$. Then $v^{\prime \prime} |\operatorname{lcm}\left(u_{1}, \ldots, u_{i-1}, v_{t}, u_{i+1}, \ldots, u_{s}, \ldots, u_{q}\right)$, for $0 \leq t \leq p$ and $v^{\prime \prime} \neq v_{t}$. Again, this implies that, if $\sigma$ is any subset of $G(\bar{I})$ such that $\left\{u_{1}, \ldots, u_{i-1}, v_{t}, u_{i+1}, \ldots, u_{s}, \ldots, u_{q}\right\} \subseteq \sigma$, for $0 \leq t \leq p$ and $v^{\prime \prime} \neq v_{t}$. Then $\sigma$ does not form a base in the minimal free resolution of $\bar{I}$. This means that the only possible case for a subset $\sigma \subseteq G(\bar{I})$ with $|\sigma|=q$ and $\left\{u_{1}, \ldots, u_{i-1}, u_{i+1}, \ldots, u_{s}, \ldots, u_{q}\right\} \subseteq \sigma$, for each $i=1, \ldots, s$, appears in the minimal free resolution of $\bar{I}$, if the following conditions hold:
\begin{itemize}
    \item $\left\{x_{1}, \ldots, x_{j}\right\} \subseteq \operatorname{supp}\left(v_{l}\right)$ for every $l=0,1, \ldots, p$, and
    \item There exists a monomial $v^{\prime}$ such that $\operatorname{supp}\left(u_{i}\right) \cap \operatorname{supp}\left(v^{\prime}\right)=\varnothing$ and $v^{\prime \prime}:=x_{1}^{d_{1}} \cdots x_{j}^{d_{j}} \cdot v^{\prime} \in \sigma$, where $\operatorname{supp}\left(u_{i}\right) \cap \operatorname{supp}\left(u_{q+1}\right)=\left\{x_{1}, \ldots, x_{j}\right\}$ and $d_{l}$ is the minimum power of $x_{l}$ appearing in the generators of $\bar{I}$.
\end{itemize}

 In addition, the subset $\sigma$ of $G(\bar{I})$ does not appear in the minimal free resolution of $\bar{I}$, where $|\sigma|>q$ and $\left\{u_{1}, \ldots, u_{i-1}, u_{i+1}, \ldots, u_{s}, \ldots, u_{q}\right\} \subseteq \sigma$, for each $i=1, \ldots, s$. This implies that the only possible case for a subset $\sigma$ of $G(\bar{I})$ appearing in the minimal free resolution of $\bar{I}$ with $|\sigma|>q$, is when it does not contain at least two of the $u_{j}$ 's, for $j \leq s$. Thus, the power of the variables appearing in the $v_{j}$ 's does not exceed the power of the variables appearing in the $u_{k}$'s. Suppose that $u_{i}, u_{j} \notin \sigma$ for $1 \leq k, j \leq s$. Therefore, 
$$
\operatorname{lcm}(\sigma) \leq \operatorname{lcm}\left(\left\{u_{1}, \ldots, u_{i-1}, v^{\prime \prime}, u_{i+1}, \ldots, u_{s}, \ldots, u_{q}\right\}\right),
$$

where $\operatorname{supp}\left(u_{i}\right) \cap \operatorname{supp}\left(u_{q+1}\right)=\left\{x_{1}, \ldots, x_{j}\right\},\left\{x_{1}, \ldots, x_{j}\right\} \subseteq \operatorname{supp}\left(v_{l}\right)$ for every $l=0,1, \ldots, p$, and $v^{\prime \prime}:=x_{1}^{d_{1}} \cdots x_{j}^{d_{j}}  \cdot v^{\prime}$ for a monomial a monomial $v^{\prime}$ such that $\operatorname{supp}\left(u_{i}\right) \cap \operatorname{supp}\left(v^{\prime}\right)=\emptyset$ and $d_{l}$ is the minimum power of $x_{l}$ appearing in the generators of $\bar{I}$. We know that $u_{q+1}$ and $v_{0}$ have the same form as $v^{\prime\prime}$. Thus, the powers of $x_{1},\ldots,x_{j}$ are less than or equal to those of $v_{0}$. Hence,

\begin{eqnarray}
&&\operatorname{reg}\left(R/{\bar{I}}\right)\nonumber\\&& \leq \max \left\{\operatorname{deg}\left(\operatorname{lcm}\left(\left\{u_{1}, \ldots, u_{i-1}, v^{\prime \prime}, u_{i+1}, \ldots, u_{s}, \ldots, u_{q}\right\}\right)\right)-q ; i=1, \ldots, s\right\}\nonumber\\&&
\leq \max \left\{\operatorname{deg}\left(\operatorname{lcm}\left(\left\{u_{1}, \ldots, u_{i-1}, v_{0}, u_{i+1}, \ldots, u_{s}, \ldots, u_{q}\right\}\right)\right)-q ; i=1, \ldots, s\right\}\nonumber.
\end{eqnarray}
This completes the proof.
\end{proof}

To make the proof of Theorem \ref{th1.14} more transparent, we present the following example.

\begin{example}
Consider the monomial ideal $I=(x_{1}^{3}x_{2}^{2},x_{3}^{4}x_{4}^{3},x_{5}x_{6}^{2},x_{1}^{2}x_{3}^{2})$. Then its integral closure has the form:\\
$\bar{I}=(x_{1}^{2}x_{3}^{2}, x_{3}^{2}x_{4}^{2}x_{5}x_{6}, x_{1}x_{3}x_{5}x_{6}, x_{1}x_{3}^{3}x_{4}^{2}, x_{1}^{2}x_{2}x_{5}x_{6}, x_{1}^{3}x_{2}x_{3}, x_{1}^{3}x_{2}^{2},x_{3}^{4}x_{4}^{3},x_{5}x_{6}^{2})$\\
We can see that $x_{1}^{2}x_{2}x_{5}x_{6}, x_{3}^{2}x_{4}^{2}x_{5}x_{6}\in G(\bar{I})$. Then $\sigma_{1}=\{x_{1}^{2}x_{3}^{2},x_{3}^{4}x_{4}^{3},x_{5}x_{6}^{2}\}$ and $\sigma_{2}=\{x_{1}^{3}x_{2}^{2},x_{5}x_{6}^{2},x_{1}^{2}x_{3}^{2}\}$ are the dominant subsets of $G(I)$ and the maximum degree of their LCMs is the regularity of the ideal. Since $x_{1}^{2}x_{2}x_{5}x_{6}|\operatorname{lcm}(\sigma)$ and  $x_{1}^{2}x_{2}x_{5}x_{6}|\operatorname{lcm}(\sigma\backslash\{x_{1}^{2}x_{3}^{2}\})$, then $\sigma$ and $\sigma\backslash\{x_{1}^{2}x_{3}^{2}\}$ do not appear in the minimal free resolution of $\bar{I}$. If a subset $\tau$ of $G(\bar{I})$ with $|\tau|\geq 3$ and containing $x_{3}^{4}x_{4}^{3}$ appears in the minimal free resolution of $\bar{I}$. Then $x_{5}x_{6}^{2}\notin \tau$. Therefore, $\operatorname{lcm}(\tau)\leq \operatorname{lcm}(\sigma)$. For example, if $\tau=\{x_{1}^{2}x_{3}^{2}, x_{3}^{2}x_{4}^{2}x_{5}x_{6}, x_{1}x_{3}x_{5}x_{6}, x_{1}x_{3}^{3}x_{4}^{2}, x_{1}^{2}x_{2}x_{5}x_{6}, x_{1}^{3}x_{2}x_{3},x_{3}^{4}x_{4}^{3}\}$, then $\operatorname{lcm}(\tau)=
x_{1}^{3}x_{2}x_{3}^{4}x_{4}^{3}x_{5}x_{6}$. Then $\operatorname{deg}[\operatorname{lcm}(\tau)]=13\leq 14=\operatorname{deg}[\operatorname{lcm}(\sigma_{1})] $. The same is true for $\sigma_{2}$. Hence, it satisfies the conditions of Remark \ref{remark1.13}. Then $\operatorname{reg}(S/\bar{I})\leq \operatorname{reg}(S/I)$. The calculation of the regularities by \textit{Macaulay2} is $\operatorname{reg}(\bar{I})=8\leq 10=\operatorname{reg}(I)$.
\end{example}

In the following theorem, we give an affirmative answer to Conjecture \ref{conj} for dominant monomial ideals.

\begin{theorem}\label{conjecture for Dominant}
Suppose that $I$ is a dominant monomial ideal in the polynomial ring $R=k[x_{1},\ldots,x_{n}]$. Then $\operatorname{reg}(\bar{I})\leq \operatorname{reg}(I)$.
\end{theorem}
\begin{proof}
Let $I=\left(u_1, \ldots, u_q\right)$ be a dominant monomial ideal, $d_{l k}=\operatorname{deg}_{x_l}\left(u_k\right)$, for $l=1, \ldots, n$ and $k=1, \ldots, q$. From \cite[Theorem 4.8]{R1}, we have $\operatorname{reg}(R / I)=\operatorname{deg}\left(\operatorname{lcm}\left(u_1, \ldots, u_q\right)\right)-q$. If $h t(I)=h t(\bar{I})=c \leq q$, then from \cite[Remark 2.3]{R3} we obtain $\operatorname{reg}(R / \bar{I})=\max \{j- \left.i \mid \beta_{i j}(R / \bar{I}) \neq 0, i \geq c\right\}$. If $x^b$ is a monomial generator of $\bar{I}$, where $b=\left(b_1, \ldots, b_n\right)$. Then from \cite[Proposition 12.1.2]{R4}, we obtain $b \in \left[\operatorname{conv}\left(\left(d_{11}, \ldots, d_{n 1}\right), \ldots,\left(d_{1 q}, \ldots, d_{n q}\right)\right)+[0,1)^{n}\right]\cap \mathbb{N}^{n}$, this implies that $b_l \leq \max \left\{d_{l k} \mid k=\right. 1, \ldots, q\}$, for every $l=1, \ldots, n$. This means that $\operatorname{deg}\left(\operatorname{lcm}\left(x^b \mid x^b \in G(\bar{I})\right) \leq\right. \operatorname{deg}\left(\operatorname{lcm}\left(u_1, \ldots, u_q\right)\right)$. Now, suppose that $\beta_{i j}(R / \bar{I}) \neq 0$, for some $i \geq c$. Then there exists a subset $\sigma \subset G(\bar{I})$ with $|\sigma|=i$ and $\operatorname{lcm}(\sigma)=j$. If $i<q$, then $\sigma$ misses at least $q-i$ the dominant generators. This implies that $\sigma$ misses at least $q-i$ dominant variables. That means $j=\operatorname{lcm}(\sigma) \leq \operatorname{deg}\left(\operatorname{lcm}\left(x^b \mid x^b \in G(\bar{I})\right)-(q-i)\right.$, then 
$$\begin{gathered}
j-i \leq \operatorname{deg}\left(\operatorname{lcm}\left(x^b \mid x^b \in G(\bar{I})\right)-(q-i)-i=\operatorname{deg}\left(\operatorname{lcm}\left(x^b \mid x^b \in G(\bar{I})\right)-q\right.\right. \\
\leq \operatorname{deg}\left(\operatorname{lcm}\left(u_1, \ldots, u_q\right)\right)-q=\operatorname{reg}(R / I).
\end{gathered}$$
If $i \geq q$, then from \cite[Remark 2.2]{R3}, we have 
$$j \leq \operatorname{deg}\left(\operatorname{lcm}\left(x^b \mid x^b \in G(\bar{I})\right) \leq\right. \operatorname{deg}\left(\operatorname{lcm}\left(u_1, \ldots, u_q\right)\right).$$
Then $j-i \leq \operatorname{deg}\left(\operatorname{lcm}\left(u_1, \ldots, u_q\right)\right)-q=\operatorname{reg}(R / I)$. Hence, $\operatorname{reg}(R / \bar{I}) \leq \operatorname{reg}(R / I)$.

\end{proof}

Based on Theorem \ref{th1.14} and Theorem \ref{conjecture for Dominant}, the following two propositions establish that Conjecture \ref{conj} holds for certain types of almost complete intersection monomial ideals.

\begin{proposition}
Suppose that $I$ is an almost complete intersection monomial ideal. If $I$ satisfies Remark \ref{remark1.13}, then $\operatorname{reg}(\bar{I})\leq \operatorname{reg}(I)$.
\end{proposition}

\begin{proof}
Suppose $I=(u_{1},\ldots,u_{q},u_{q+1})$ with $u_{1},\ldots,u_{q}$ is a regular sequence. If $u_{q+1}$ does not divide the least common multiple of any subset $\sigma$ of $G(I)$, then $I$ is a dominant monomial ideal and by Theorem \ref{conjecture for Dominant} we have the result. Now, suppose that there exists $2\leq s \leq q$ such that $u_{q+1}|\operatorname{lcm}(u_{1},\ldots,u_{s})$. Thus, from Theorem \ref{th1.3} we have $\operatorname{reg}(S/I)=\max\{\operatorname{deg}(\operatorname{lcm}(u_{1},\ldots,u_{k-1},u_{q+1},u_{k+1},\ldots,u_{s},\ldots,u_{q}))| k=1,\ldots,s\}-q$. According to Theorem \ref{th1.14}, we have 
\begin{eqnarray}
&&\operatorname{reg}(R/\bar{I})\leq\nonumber\\&&\max\{\operatorname{deg}(\operatorname{lcm}(u_{1},...,u_{k-1},u_{q+1},u_{k+1},\ldots,u_{s},\ldots,u_{q}))| k=1,\ldots,s\}-q=\operatorname{reg}(R/I).\nonumber
\end{eqnarray}
This completes the proof.
\end{proof}

\textbf{Conflict of interest.}\\ The author declares that there is no conflict of interests.



\begin{thebibliography}{99}
\bibitem{R1} G. Alesandroni, \textit{Minimal resolutions of dominant semidominant ideals}, J. Pure and Appl. Algebra, \textbf{221}(2017), 780-798.

\bibitem{R16} D. Bayer,  I. Peeva and B Sturmfels, \textit{Monomial resolutions}, Math. Res. Lett., \textbf{5}(1998), 31-46.

\bibitem{R11} Y. Cui, C. Gong and G. Zhu, \textit{The regularity of equigenerated monomial ideals and their integral closures}, arXiv:2509.15119, (2025).

\bibitem{R9} D. R. Grayson and M. E. Stillman, \textit{Macaulay 2, a software system for research in algebraic geometry}, Available at http://www.math.uiuc.edu/Macaulay2/.

\bibitem{R14} E. Guardo and A. Van Tuyl, \textit{Powers of complete intersections: graded Betti numbers and applications}, Ill. J. Math., \textbf{49}(2005), 265-279.

\bibitem{R15} C. Huneke and I. Swanson,  \textit{Integral closure of ideals, rings, and modules}, \textbf{13}, Cambridge University Press, (2006).

\bibitem{R3} O. Javadekar, \textit{A comparison of the regularity of certain classes of monomial ideals and their integral closure}, Arch. Math., \textbf{126}(2026), 1-13.

\bibitem{R10} A. Küronya and N. Pintye, \textit{Castelnuovo--Mumford Regularity and Log-canonical Thresholds}, arXiv:1312.7778, (2013).

\bibitem{R7} G. Lyubeznik, \textit{A new explicit finite free resolution of ideals generated by monomials in an R-sequence}, J. Pure and Appl. Algebra, \textbf{51}(1988), 193-195.

\bibitem{R6} A. Mafi and R. R. Qadir, \textit{Betti numbers and almost complete intersection monomial ideals}, arXiv:2505.18788, (2025).

\bibitem{R13} M. Mandal and S. Priya,  \textit{Bounds on the Castelnuovo-Mumford Regularity in dimension two}, arXiv:2404.01684, (2024).

\bibitem{R8} J. Mermin, \textit{Three simplicial resolutions}, arXiv:1102.5062, (2011).

\bibitem{R2} S. Misra, \textit{A counterexample to a conjecture of Küronya and Pintye on regularity and integral closure}, arXiv:2605.13879.

\bibitem{R5} I. Peeva, \textit{Graded syzygies},  \textbf{14}, Springer Science \& Business Media, (2010).

\bibitem{R12} M. E. Rossi,  D. T. Trung and N. V. Trung, \textit{Castelnuovo–Mumford regularity and Ratliff–Rush closure}, J. Algebra, \textbf{504}(2018), 568-586.

\bibitem{R17} H. Scarf, \textit{The Computation of Economic Equilibria, Cowles Foundation Monograph}, \textbf{24}, Yale University Press, (1973).

\bibitem{R4} R. H. Villarreal, \textit{Monomial Algebras, Monographs and Research Notes in Mathematics}, Chapman and Hall/CRC, (2015).

\end{thebibliography}
\end{document}